\documentclass[11pt,a4paper]{amsart}

\usepackage{latexsym}

\usepackage{amssymb}
\usepackage{amsthm}
\usepackage{amsmath}
\usepackage{amstext}
\usepackage{amscd}
\usepackage[all]{xy}
\usepackage{color}

\numberwithin{equation}{section}

\theoremstyle{plain}
\newtheorem{thm}{Theorem}[section]
\newtheorem{prop}[thm]{Proposition}
\newtheorem{cor}[thm]{Corollary}
\newtheorem{lem}[thm]{Lemma}

\newtheorem{theorem*}{Theorem}[]

\theoremstyle{definition}

\theoremstyle{remark}
\newtheorem{rem}[thm]{Remark}

\newcommand{\C}{\mathbb{C}}

\newcommand{\R}{\mathbb{R}}
\newcommand{\K}{\mathbb{K}}

\newcommand{\proj}{\mathbb{P}}

\DeclareMathOperator{\Int}{interior}

\newcommand{\ps}{\psi}

\newcommand{\Ps}{\Psi}
\newcommand{\Ph}{\Phi}
\newcommand{\f}{f}

\title{Algebraic Stratified General Position and Transversality}

\author{ Clint McCrory, Adam Parusi\'nski, Lauren\c tiu P\u aunescu }

\begin{document}
\address {Mathematics Department, University of Georgia, Athens GA
30602, USA}

\email{clint@math.uga.edu}

\address {Univ. Nice Sophia Antipolis, CNRS,  LJAD, UMR 7351, 06108 Nice, France}

\email{adam.parusinski@unice.fr}
\address{School of Mathematics, University of Sydney,
  Sydney, NSW, 2006, Australia }%
\email{laurent@maths.usyd.edu.au}%

\thanks{Partially supported  by ANR project STAAVF (ANR-2011 BS01 009), by ARC DP150103113, by Sydney University BSG2014 and by Faculty of Science PSSP}


\begin{abstract}
The method of Whitney interpolation is used to construct, for any real or complex projective algebraic variety, a stratified submersive family of self-maps that yields stratified general position and transversality theorems for semialgebraic subsets.

 \end{abstract}

\maketitle

In classical algebraic topology, general position of chains was used by Lefschetz to define the intersection pairing on the homology of a manifold. The stratified general position of the first author \cite{Mc} was used by Goresky and MacPherson to define the intersection pairing on the intersection homology of a complex algebraic variety. Murolo, Trotman, and du Plessis \cite{MTdP} proved a stratified transversality theorem for the most important types of regular stratifications, and they used their result to define cup and cap product for certain stratified homology theories.

Here we prove stratified general position and transversality theorems for semialgebraic subsets of algebraic stratifications. We use the methods of \cite{PP}, which involve Zariski equisingularity and Whitney interpolation. Our theorems can be used to define the intersection pairing for the intersection homology of semialgebraic chains. (In their initial paper \cite{GM}, Goresky and MacPherson considered  piecewise linear chains with respect to a triangulation.) In a subsequent paper \cite{MP1} we will apply our transversality theorem to define an intersection pairing for \emph{real intersection homology}, an analog of intersection homology for real algebraic varieties.


\section{Submersive families}\label{submersive}

Let $T$ and $M$ be smooth (that is, $C^1$) manifolds  and let $\Ps : T\times M  \to M$ be a smooth mapping. 
Consider  $\Ps_t : M \to M$, $\Ps_t(x) = \Ps (t, x)$, and $\Ps ^x : T \to  M$, $\Ps ^x(t) = \Ps (t, x)$.   We say  $\Ps$ is a family of diffeomorphisms if for all $t\in T$ the map $\Ps_t$ is a diffeomorphism. The family $\Ps $ is called \emph{submersive} if, for each $(t, x) \in T\times M$, the differential $D\Ps ^x$ at $t$ is surjective (\emph{cf}.\ \cite{SMT} I.1.3.5).
For a stratified set $X= \bigsqcup S_i$ we say that $\Ps  : T \times X \to X$ is \emph{a stratified submersive family of diffeomorphisms} 
if for each stratum $S_j$,  we have $\Ps(T\times S_j)\subset S_j$, and the map $\Ps : T \times S_j  \to S_j$ is a submersive family of diffeomorphisms.  

By an algebraic, resp. projective, variety we mean an algebraic, resp. projective, variety over 
$\K= \C$ or $\R$.  
 All sets and functions are supposed to be at least semialgebraic (in the real algebraic sense).    
 Our main goal is to show the following theorem. 

\begin{thm}\label{projective}
Let $\mathcal V = \{V_i\}$ be a finite family of algebraic subsets of projective space $\proj^n$. There exists an algebraic stratification $\mathcal S = \{S_j\}$ of $\proj^n$ compatible with each $V_i$ and a semialgebraic stratified submersive family of diffeomorphisms  $\Ps : U\times \proj^ n  \to \proj^n$,  where $U$ is an open neighborhood of the origin in $\K^{n+1}$, 
such that $\Ps(0,x)=x$ for all $x\in \proj^n$.  Moreover, the map $\Ph:U\times \proj^ n  \to U\times \proj^n$, $\Ph(t,x) = (t,\Ps(t,x))$, is an arc-wise analytic trivialization of the projection $U\times \proj^n\to U$.
\end{thm}

\begin{rem}
A similar result holds for affine varieties; see Corollary \ref{affine}.
\end{rem}

By a \emph{stratification} of a semialgebraic set $X$ we mean a locally finite partition $\mathcal S$ of $X$ into smooth semialgebraic manifolds called the \emph{strata} of $\mathcal S$.
An  \emph{algebraic stratification}  of an algebraic variety $X$ is a stratification coming from a filtration of $X$ by algebraic subvarieties 
$$X=X_n\supset X_{n-1} \supset \cdots \supset X_0 \supset X_{-1} = \emptyset $$
    such that for each $j$, either $\dim X_j = j$ or $X_j = X_{j-1}$, and $\operatorname{Sing}(X_j) \subset X_{j-1}$.  
This filtration induces a decomposition $X=\bigsqcup S_i$, where the $S_i$ are the connected components of all $X_j\setminus  X_{j-1}$.  The sets $S_i$ are the strata of the algebraic stratification $\mathcal S= \{S_i\}$ 
of $X$.   If $\K=\R$ then each stratum $S_j$ is  a nonsingular semialgebraic subset of $X$.  If $\K=\C$ then each stratum $S_j$  is a nonsingular  locally closed subvariety of $X$. (The stratifications we construct in the proof of Theorem \ref{projective}  satisfy the \emph{frontier condition}: If $S_1$ and $S_2$ are disjoint strata such that $S_1$ intersects the closure of $S_2$, then $S_1$ is contained in the closure of $S_2$.)  

Arc-wise analytic trivializations were introduced in \cite{PP}.  In detail the family $\Ps$ of Theorem \ref{projective} 
has the following properties.  It is continuous and semialgebraic and preserves the strata; that is, for each stratum $S$, $\Ps$ induces a map $\Ps_S:U\times S\to S$ (we often skip the subscript $S$ if there is no confusion).  Moreover, the map $\Ph_S(t,x)= (t,\Ps_S(t,x)): U\times S \to U\times S$ 
is a real analytic isomorphism.  The arc-wise analyticity means that $\Ps$ is analytic in $t$ and 
for every real analytic arc $x(s):(-1,1) \to \proj^n$, $(t,s) \mapsto \Ps(t,x(s))$ is analytic ($\K$-analytic in $t$ and real analytic in $s$ ).  In particular, $\Ph(t,x) =(t, \Ps (t,x))  : U\times \proj^ n  \to U\times \proj^n$   
is a semialgebraic homeomorphism that is real analytic on  real analytic arcs.  Finally, we require that 
$\Ph^{-1}$ is also real analytic on  real analytic arcs. 

We give two useful properties of a stratified submersive family of diffeomorphisms. Proposition \ref{generalposition} is a stratified general position theorem, and Proposition \ref{transversality} is a stratified transversality theorem. Although transversality implies general position, we first present a simple direct proof of general position.

The following result is an algebraic version of the piecewise linear general position theorem of \cite{Mc} (\emph{cf}. \cite{GM}, p.\ 142).

\begin{prop}\label{generalposition}
Let $\Ps : U\times \proj^ n  \to \proj^n$ be a stratified family as in Theorem \ref{projective}, and let $\mathcal S$ be the associated algebraic stratification of $\proj^n$.  
 Let $Z$ and $W$ be semialgebraic subsets of $\proj^n$. There is an open dense semialgebraic subset $U'$ of $U$ such that, for all $t \in U'$ and all strata $S\in \mathcal S$,
 $$
\dim   (Z \cap \Ps_t^{-1} (W )\cap S) \le  \dim (Z\cap S)  + \dim (W\cap S)  - \dim S. 
$$
\end{prop}

This proposition is a consequence of the following lemma. 

\begin{lem}
Let $T$ and $M$ be smooth (equidimensional) semialgebraic subsets of $\proj^n$, with $\dim T \ge  \dim M$. Let $\Ps : T\times M \to M$ be a semialgebraic submersive family of diffeomorphisms. Let $Z$ and $W$ be semialgebraic subsets of $M$. 
There is an open dense semialgebraic subset $T'$ of $T$ such that, for all $t \in T'$,
$$
\dim(Z\cap \Ps_t^{-1} (W ))\le  \dim Z + \dim W - \dim M. 
$$
\end{lem}

\begin{proof} 
Consider the map $\Theta:T\times M\to M\times M$, $\Theta(t,x) = (x,\Ps (t,x))$.
Since $\Ps$ is a submersive family,  $\Theta$ is a submersion.  Therefore for all $(x,y)\in M\times M$, if $\Theta^{-1}(x,y)\neq \emptyset$ then $\dim \Theta^{-1}(x,y) = \dim T - \dim M$.  Thus, because
$(T\times Z  )\cap \Ps^{-1}(W ) = \Theta^{-1} (Z\times  W)$, we have 
$$
\dim((T\times Z  ) \cap \Ps^{-1}(W ) ) \le  \dim Z  + \dim W - \dim M + \dim T.  
$$
Let $Q = \{t \in T\ ;\ \dim Z\cap \Ps_t^{-1} (W)  \le  \dim Z + \dim W - \dim M \}$. 
If $Q$ is not dense, there exists an open subset $\Omega$ of $T$ such that, for all $t \in \Omega$,
$$ 
\dim (Z\cap \Ps _t ^{-1}(W )) > \dim Z + \dim W - \dim M.
$$
Thus
$$ 
\dim ((T\times Z )\cap \Ps  ^{-1}(W))   > \dim Z + \dim W - \dim M + \dim T,
$$
which is a contradiction.  The set $Q$ is semialgebraic and dense, so we can take $T' = \Int Q$.
\end{proof}

The following result is an algebraic version of the transversality theorem of \cite{MTdP} (Theorem 3.8, p.\ 4887). If $\mathcal S$ is a stratification of a semialgebraic set $X$, and $\mathcal T$ is a stratification of a semialgebraic subset $Y$ of $X$, then $(Y,\mathcal T)$ is a \emph{substratified object} of $(X,\mathcal S)$ if each stratum of $\mathcal T$ is contained in a  stratum of $\mathcal S$. Two substratified objects $(Z,\mathcal A)$ and $(W,\mathcal B)$ of $(X,\mathcal S)$ are   \emph{transverse} in $(X,\mathcal S)$ if, for every pair of strata $A\in\mathcal A$ and $B\in\mathcal B$ such that $A$ and $B$ are contained in the same stratum $S\in\mathcal S$, the manifolds $A$ and $B$ are transverse in $S$.

\begin{prop}\label{transversality}
Let $\Ps : U\times \proj^ n  \to \proj^n$ be a stratified family as in Theorem \ref{projective}, and let $\mathcal S$ be the associated algebraic stratification of $\proj^n$.  
 Let $Z$ and $W$ be semialgebraic subsets of $\proj^n$, with  semialgebraic stratifications $\mathcal A$ of $Z$ and $\mathcal B$ of $W$ such that $(Z,\mathcal A)$ and $(W,\mathcal B)$ are substratified objects of $(\proj^n,\mathcal S)$. There is an open dense semialgebraic subset $U'$ of $U$ such that, for all $t \in U'$, $(Z,\mathcal A)$ is transverse to $\Ps_t^{-1} (W,\mathcal B)$ in $(\proj^n,\mathcal S)$,
\end{prop}

This proposition is a consequence of the following lemma. 

\begin{lem}Let $T$ and $M$ be smooth (equidimensional) semialgebraic sets, with $\dim T \ge  \dim M$. Let $\Ps : T\times M \to M$ be a semialgebraic submersive family of diffeomorphisms. Let $A$ and $B$ be smooth semialgebraic subsets of $M$. 
There is an open dense semialgebraic subset $T'$ of $T$ such that, for all $t \in T'$, $A$ is transverse to $\Ps_t^{-1} (B )$.
\end{lem}
\begin{proof} We apply a standard transversality technique (\emph{cf.}\ \cite{SMT} Theorem I.1.3.6, p.\ 39). Since $\Ps$ is a submersive family, the map $\Theta:T\times M\to M\times M$, $\Theta(t,x) = (x,\Ps (t,x))$, is a submersion. So $\Theta^{-1}(A\times B) = (T\times A)\cap \Ps^{-1}(B)$ is a smooth submanifold of $T\times M$. Let $\pi: T\times M\to T$ be the projection, and let
$$
P = \{ t\in T\ ;\ t\ \text{is a critical value of}\ \pi | (T\times A)\cap\Ps^{-1}(B) \}.
$$
Now $t\in P$ if and only if there exists $x\in A$ such that $(t,x)\in  (T\times A)\cap\Ps^{-1}(B)$ and
$$
\bold T_{(t,x)}(\{t\}\times A)+\bold T_{(t,x)}((T\times A)\cap\Ps^{-1}(B))  \neq \bold T_{(t,x)}(T\times A),
$$
where $\bold T$ denotes the tangent space. This is equivalent to
$$
\bold T_{(t,x)}(\{t\}\times A)+\bold T_{(t,x)}(\Ps^{-1}(B)) \neq \bold T_{(t,x)}(T\times M);
$$
in other words, $\{t\}\times A$ is not transverse to $\Ps^{-1}(B)$. 
This is equivalent to the condition that $\Ps_t: A\to M$ is not transverse to $B$; \emph{i.e.}\  $A$ is not transverse to $\Ps_t^{-1}(B)$.

By Sard's theorem, $P$ has measure zero in $T$. (Stratify the semialgebraic map $\pi: (T\times A)\cap\Ps^{-1}(B)\to T$ so that the restriction of $\pi$ to each stratum is real analytic. Then apply Sard's theorem to these restrictions.) Therefore $Q = T\setminus P$ is dense in $T$. Since $Q$ is semialgebraic, we can let $T'$ be the interior of $Q$.
\end{proof}

\begin{rem}
In the application of our general position and transversality theorems to intersection theory for singular varieties, we do not encounter the problems described by Murolo \emph{et al.} (\cite{Mu} \S5; \cite {MTdP} p.4892),
for the simple reason that the intersection of two semialgebraic sets is semialgebraic.
\end{rem}


\section{Whitney interpolation}

In this section we adapt the Whitney Interpolation of \cite{PP} 
 and define the interpolation functions by modifying the function defined  in Example A.7 of \cite{PP} by adding a parameter $\eta\in \C$.  

 Given two subsets $\{a_1, \ldots, a_N\}\subset \C$, $\{b_1, \ldots, b_N\}\subset \C$,   
 such  that if $a_i=a_j$ then $b_i=b_j$ and if $a_i=0$ then $b_i=0$,    we will define a 
Lipschitz  interpolation function $\ps : \C\to \C$, depending continuously on $a,b$ and a 
 parameter $\eta\in \C$,  such that $\ps (a_i)=b_i$ and $\ps (0) =0$.  Moreover, if we set 
\begin{align}
\gamma = \max_{a_i\ne a_j} \frac { |D_i - D_j|}{|a_i -a_j|}, 
\end{align}
where  $ D_i = b_i - a_i$, then we will show that  $\ps $ is a bi-Lipschitz homeomorphism provided $\gamma$
and $\eta$ are sufficiently small.  

Denote by $\sigma_i = \sigma_i (\xi_1, \ldots ,\xi_N)$ 
the elementary symmetric functions in  $\xi_1, \ldots ,\xi_N \in \C$.    
Let  $P_k  =  \sigma^{\alpha_k}_k  $, where $\alpha_k =  (N!)/k$, so that all $P_k$ are homogeneous of the same degree.  Then we define 
\begin{align}\label{fi3}
\f_j (\xi) =  \frac 1 {N!}  \sum _k  \xi_j  \frac {\partial  P_k}{\partial \xi_j} (\xi )  \overline {P_k(\xi )} ,  
\end{align} 
where $\xi= (\xi_1, \ldots ,\xi_N )$ and the bar denotes  complex conjugation, and 
\begin{align}\label{fi4}
\f (\xi) =\sum f_j(\xi)=  \sum _k  P_k(\xi)  \overline {P_k (\xi)} .  
\end{align}
Note that $f$ is real-valued,  $\R$-homogeneous of degree $d=2N!$, and vanishes only at the origin. 
Let 
$$\mu_i(z) :=\f _i((z-a_1)^{-1} , \ldots, (z-a_N)^{-1} )  , \quad \mu (z) 
:=\f ((z-a_1)^{-1} , \ldots, (z-a_N)^{-1} )  .$$
Define the  interpolation map $\ps: \C \to \C$ (parametrized by $a,b$ and $\eta\in \C$) by 
the formula 
\begin{align}\label{psireal} 
\ps(z) =\ps (z,a,b, \eta) = z + \frac {\sum_{i=1}^N  \mu_i(z)  (b_i-a_i) + \eta z |z|^{- 2N!}}
 {\mu(z) +  |z|^{- 2N!}} 
\end{align}
and set formally $\ps (a_i) = b_i$, $\ps (0)=0$ (as we show later, this is 
the extension by continuity).   Note that 
\begin{align}\label{lambda} 
\ps(\lambda z, \lambda a, \lambda b, \eta) = \lambda \ps (z,a,b, \eta), \quad \text{ for all } \lambda\in \C .
\end{align}
We may consider  $\ps$ as a function of variables $z,a,b,\eta$ defined on 
\begin{align*}
\Xi=\{(z,a,b, \eta)\in \C\times \C^N \times \C^N \times \C\ ;\ \text { if } a_i=a_j \text { then } 
b_i=b_j , \text {  if } a_i=0\text { then } 
b_i=0 \}.  
\end{align*}

\begin{prop}\label{psiproperties}
The function $\ps$ satisfies: 
\begin{enumerate}
\item
$\ps$ is continuous  on $\Xi$.  
\item
There is a universal constant $C=C(N)$ such that  the map $\ps: \C\to \C$ is Lipschitz with Lipschitz constant $1+ C(\gamma +|\eta |)$.  If $C(\gamma +|\eta |)< 1$ then $\ps: \C\to \C$ is a bi-Lipschitz homeomorphism,   with $(1-C(\gamma +|\eta |))^{-1}$ as  a Lipschitz constant of $\ps^{-1}$.    
\item
$\partial \ps/\partial \eta (z,a,b,\eta) = 0 $ if and only if $z=a_i$ or $z=0$. 
\end{enumerate}
\end{prop}
 
\begin{proof}
Sometimes it will be convenient to use  the following formula for $\ps$:
\begin{align}\label{2ndformulaforpsi}
\ps (z)= z + \frac {\sum_{i=1}^N  \mu_i(z)  (b_i-a_i) |z|^d+ \eta z }
 {\mu(z) |z|^d+  1}, \quad d=2N!.
 \end{align}

We begin the proof of (1) by showing that $\ps (z,a,b,\eta) \to 0=\ps (0, a_0, b_0, \eta_0)$ if $(z,a,b,\eta) \to (0, a_0, b_0, \eta_0)$ (later we suppress  $a,b,\eta$ in the notation when it causes no confusion).   Since the denominator of the fraction in \eqref{2ndformulaforpsi} is always real positive $\ge 1$, $\frac { \eta z }  {\mu(z) |z|^d+  1} \to 0$ as $z\to 0$.   
If $a_{0i} \ne 0$ then  $ \mu_i(z)  (b_i-a_i) |z|^d\to 0$ and hence $\frac { \mu_i(z)  (b_i-a_i) |z|^d}  {\mu(z) |z|^d+  1} \to 0$ as $z\to 0$.  Finally consider the 
case $a_{0i} = 0$.  Then $b_{0i}=0$ and hence $b_i-a_i \to 0$.  Since in general 
$|\mu_i|\le  C_1 \mu$ for a universal constant $C_1$, by conditions (3) and (5) of Appendix A of \cite{PP}, 
we have $\frac { \mu_i(z)  (b_i-a_i) |z|^d}  {\mu(z) |z|^d+  1} \to 0$ as $z\to 0$ also in this case.  
 
The remaining part of the  proof of (1) is similar to the proof of Proposition A.4 of \cite{PP}.  
Let $(z,a,b,\eta) \to (z_0, a_0, b_0, \eta_0)$, $z_0\ne 0 $.   We suppose that $z_0$ equals one of the $a_{0i}$;  otherwise the proof is easy.   
Thus let $z_0=a_{01} \ne 0 $.  Then $\ps (z_0,a_0,b_0, \eta)  = b_{01}$ independently of $\eta$.  
Denote   $J=\{j\in \{1, \dots ,n\}\ ;\ a_{0j}= a_{01}\}$.    Then 
\begin{align*}
 \ps (z,a,b, \eta)  - \ps (z_0,a_0,b_0, \eta_0 )  & = (z - z_0)+  
\frac { \sum_{i}  \mu_i (z,a) (b_i  -a_i)+ \eta z|z|^{-d} } { \mu (z,a) +  |z|^{- d}} - (b_{01}-a_{01}) \\
& = (z - z_0)+  
\frac { \sum_{i\in J}  \mu_i (z,a) ((b_i -b_{01}) -(a_i-a_{01}))} { \mu (z,a) +  |z|^{- d}} \\
& + \frac { \sum_{i\notin J} \mu_i (z,a) ((b_{i}-b_{01}) -(a_{i}- a_{01}))} { \mu (z,a) +  |z|^{- d}}    + \frac {  \eta z- (b_{01 }- a_{01})  } { \mu (z,a)  |z|^{d}+1 }  . 
\end{align*}
The first three summands converge to $0$ as $(z,a,b) \to (z_0, a_0, b_0)$ by the arguments given in    \cite{PP}, and the last one because $\mu (z,a) \to \infty$.  

The proof of (2) is similar to that of Proposition A.3 of \cite{PP}.  
It relies on the formula (A.7)  of  Appendix A of \cite{PP}, which says that there are universal constants $A, B, D$ such that 
\begin{align}\label{boundonmu}  \notag
& |\mu_i(z)|\le  A  |z-a_i|^{-1}  \mu (z) ^{1- \frac {1} d} ,  \\
& |\mu_i '(z)| \le B  |z-a_i|^{-1}  \mu (z),  \\ \notag
&|\mu '(z)| \le D \mu (z)^{\frac {d+1} d}, 
\end{align}
where ``prime'' denotes any directional derivative (in the variable $z$) $\frac {\partial } {\partial v} $, $v\in \C$, $|v|=1$. 

Given $z\in \C$, choose $j$ such that $|z-a_j| = \min _i |z-a_i|$.  Then, for all $i$,
\begin{align}\label{boundfordifference}
|a_i-a_j|\le 2 |z-a_i|.
\end{align}
Denote $I_j = \{i; a_i=a_j\}$.   By differentiating,
\begin{align} \notag
 | (\ps (z) - z)'| \le & \frac {\sum_{i\notin I_j} |\mu'_i (z) ( D_i- D_j)| |z|^d+ \sum_{i\notin I_j} |\mu_i (z) ( D_i- D_j)| (|z|^d)' + |\eta|} {|z|^d \mu (z)+1} \\ \notag
 +  & \frac {(\sum_{i\notin I_j} |\mu_i (z) ( D_i- D_j)| |z|^d + |\eta z|) (|z|^d ||\mu' (z)|+ (|z|^d)'|\mu|)} {( |z|^d\mu (z)+1)^2} .  
\end{align}
We have 
\begin{align}\label{Dbound}
|D_i - D_j| \le \gamma |a_i -a_j|\leq 2\gamma |z-a_i|.  
\end{align}

Using \eqref{Dbound} 
and  \eqref{boundonmu} we get
\begin{align*}
 |\mu'_i (z) ( D_i- D_j)| \le 2 B \gamma \mu (z) 
\end{align*}
and 
 \begin{align*} 
|\mu_i (z) ( D_i- D_j)|  |\mu' (z)| \le 2 AD \gamma( \mu (z))^2 .  
\end{align*}
Moreover  $|(|z|^d)'| \mu ^{1- \frac{1}{d} }(z) \leq d(|z|^d \mu  (z) +1)$.  Indeed, this follows from the fact that $a^{d-1} \le (a^d+1)$ for all $a\ge 0$ applied to $a=|z| \mu^{1/d}  (z)$.   All this shows that 
$$\frac {\sum_{i\notin I_j} |\mu'_i (z) ( D_i- D_j)| |z|^d+ \sum_{i\notin I_j} |\mu_i (z) ( D_i- D_j)| (|z|^d)' + |\eta|} {|z|^d \mu (z)+1} \leq (2NB+2dNA)\gamma+|\eta|.
$$
Similarly we obtain 
$$ \frac {(\sum_{i\notin I_j} |\mu_i (z) ( D_i- D_j)| |z|^d + |\eta z|) (|z|^d |\mu' (z)|+ (|z|^d)'\mu)} {( |z|^d\mu (z)+1)^2}  \leq  (2dNA+2NAD)\gamma+d|\eta|+D|\eta|   .  $$
This finally shows that 
$$ | (\ps (z) - z)'|  \leq C(\gamma +|\eta|),$$
 for a universal constant $C=C(N)$, i.e.  $ z \to \ps (z) - z$ is a Lipschitz contraction with Lipschitz constant $C$.  This shows (2); see  the proof of Proposition A3 of \cite{PP} 
for details. 

The claim (3) follows from the fact that 
\begin{align}
\frac {\partial \ps}   {\partial \eta} = \frac {z} {|z|^d \mu (z)+1} 
\end{align}
vanishes if $z=0$ and if $z=a_i$.  In the latter case the denominator equals infinity.  
\end{proof}

Now we show the arc-wise analyticity of $\ps$. In the applications (see the next section), 
the entries of the vectors $a$ and $b$ are the roots of a polynomial.  Note that 
$\ps$ is symmetric in $a$ and $b$ simultaneously; that is, $\ps$ is invariant if we apply the same 
permutation to the entries of $a$ and of $b$.  

The space $\C^N\ni a$ can be stratified by the \emph{type} of $a=\{a_1,\dots,a_N\}$; that is, by the number of distinct $a_i$ and the multiplicities with which they appear in the vector $a$. These multiplicities $m_s$, $s=0, \dots , k$  are defined as follows.  Let $W_0 = \{i\ ;\ a_i = 0 \}$. Define an equivalence relation on the set $\{1,\dots, N\}\setminus W_0$ by $i\sim j$ if $a_i = a_j$. Let $W_1,\dots,W_k$ be the collection of equivalence classes, a partition of the set $\{1,\dots, N\}\setminus W_0$. Let $W = \{W_0,W_1,\dots,W_k\}$, and let $m_s = |W_s|$. By definition $m_0\geq 0$ and $m_s>0$ for $s>0$.    We encode such a type by the multiplicity vector   
 $\mathbf m= (m_0, m_1, \dots , m_k)$, with $\sum _{s=0}^k m_i=N$. We choose the index set so that $m_1\le \dots \le m_k$.   

We denote by  $S_{\mathbf m}\subset \C^N$ the  set of $a$ with  multiplicity vector $\mathbf m$.  The strata of $\C^N$ are the connected components of the sets $S_{\mathbf m}$. Two vectors $a, a'\in S_{\mathbf m}$ are in the same stratum if the associated sets $W_0$ and $W_0'$ are equal and the associated partitions $\{W_1,\dots,W_k\}$ and $\{W_1',\dots,W'_{k'}\}$ are equal.

Note that the stratification by the type of $a$ can be defined by the signs of symmetric  polynomials 
in the entries of $a$.  
More precisely, for a multiplicity vector   $\mathbf m=(m_0,\dots,m_k)$, a connected component of 
$S_{\mathbf m}$ is a connected component of the set $T\supset S_{\mathbf m}$ defined as follows. Let $D_j$ be the classical generalized discriminants; see  \cite{whitneybook} Appendix IV or  \cite {PP} Appendix B. Let $l$ be the number of nonzero entries of $\mathbf m$. If $m_0 > 0$ then $l=k+1$ and $T$ is given by 
\begin{equation}\label{disc}
D_{l+1} = \dots = D_N=0,\ D_l\ne 0
\end{equation}
and $\prod a_i=0$.  
If $m_0 = 0$ then $l=k$ and $T$ is given by the condition (\ref{disc}) and $\prod a_i\neq 0$.

\begin{prop}\label{psiproperties2}
The function $\ps$ satisfies:  
\begin{enumerate}
\item  [(4)] 
$\ps (z,a,b, \eta )$ is a polynomial of degree $1$ in $b$ and $\eta$,  and for every stratum $S$ of the stratification  by the type of $a$,  $\ps$ is  real analytic on $ \C \times S \times S\times \C $.  
\item [(5)] 
   Let  an unordered set of function germs $b(t,s)=(b_1(t,s), \dots , b_{N}(t,s))$, $(t,s) \in  (\K^m\times \R ,(0,0))$, be such that  for every  symmetric polynomial 
   $G$  in $b$,  $G (b(t,s))$ is analytic  in $(t,s)$   (it equals a power series in $(t,s)\in \K^m\times \R$).  We also assume that the type of $b(t,s)$ is independent of $t$ and that the last not identically equal to zero generalized discriminant 
 $D_l (b_1(t,s),\dots , b_{N}(t,s))$ is of the form $s^k u(t,s)$, $k\ge 0$, $u(0,0)\ne 0$, and 
 we make a similar assumption on the product of $b_i$.    \\
Let  $z(s): (\R,0) \to \C$  be  a real analytic germ and  set $a(s) =   b(0,s)$.  Then $\ps (z(s),a(s),b(t,s), \eta)$  is analytic in $(t,s, \eta)$.   
\end{enumerate}
\end{prop}

\begin{proof} 
The proof of (4) is similiar  to that of Lemma 2.7  of \cite{PP}.  Let us just sketch it.  
 Let $S=S_W$ for a partition $W$.  We consider only the case $m_0\ne 0$; the case $m_0=0$ 
is similar.  Choose the representatives $i_0, \dots ,i_k$ so that  $i_s\in W_s$.  
 Write similarly to (2.7) of \cite{PP}  
 \begin{align}\label{ourpsi3} 
\ps (z,a,b, \eta)  =  z + \frac { Q(z,a) \overline Q(z,a) \bigl (  \sum _k  \sum_{j}
  \ Q_{k,j} (z,a)  \overline  Q_k (z, a)   ( b_j - a_j)|z|^d +\eta z  \bigr) } 
{N!  Q(z,a) \overline Q(z,a)  (  \sum _k  Q_k(z,a)  \overline  Q_k (z,a )|z|^d +1) } ,
\end{align}
where 
\begin{align*}
& Q_k(z,a) = P_k ((z-a_1)^{-1} , \dots , (z-a_N)^{-1}) , \\
& Q_{k,j}(z,a,b) =
(z-a_j)^{-1} \frac {\partial  P_k}{\partial \xi_j} ((z-a_1)^{-1} , \dots , (z-a_N)^{-1}),  \\
 & Q(z,a) = Q_W(z,a) =  \prod_{s=0}^k (z-a_{i_s})^{N!} \,  .
\end{align*}
Then the denominator of the fraction in \eqref{ourpsi3}   does not vanish 
on $S$, and both the numerator and the 
denominator are real analytic on  $ \C \times S \times S\times \C $.  


The proof of (5) is similar to that of Lemma  2.8 of \cite{PP}, but because of the presence of the term 
$|z|^{2N!}$ we cannot apply the reduction to the case $z(s) \equiv 0$ by 
 subtracting $z(s)$ from every component of $b(t,s)$.  
 
  We consider $b_i (t,s)$ as the roots of a polynomial 
$
 G(z,t,s) = z^N + \sum_{i=1}^N c_i(t,s) z^{N-i} 
$
 with coefficients analytic in $t$ and $s$, that is as  power series in $(t,s)\in \K^m\times \R$.  
 The next step is to complexify the variables and  consider $c_i(t,s)$ as complex analytic germs of 
 $(t,s)\in (\C^m\times \C, (0,0))$.  By the assumption on the generalized discriminants and on the product 
 of $b_i$ the type of $b_i(t,s)$ is independent of $t$ also for $s$ complex.  Moreover, by the assumption on the 
generalized  discriminants, we may apply  to  $G$ the Puiseux with parameter theorem, see \cite{PP} \S 2.   In particular, for $s$ fixed, 
 an ordering of the roots $a_1(s), \dots ,a_N(s)$ of $G(z,0,s)$ gives, by continuity in $t$, 
 an ordering of the roots  $b_1(t,s), \dots ,b_N(t,s)$ of $F$.  Fix such an ordering and define 
 $$
\varphi (t,s) = \ps (z(s), a(s), b(t,s)),
$$
where $\ps$ is given by \eqref{ourpsi3}. ($Q$ of \eqref{ourpsi3} depends on the choice of $W$.  
We take the one given by the type of $b(t,s)$, $s\ne 0$.)  Thus defined, $\varphi$ is independent of the choice of 
the initial ordering of $a_1(s), \dots ,a_N(s)$.  Since $P_k (\xi) $ is symmetric in $\xi $, $Q(z(s),a(s)) $ and the product 
$Q(z(s),a(s)) Q_k(z(s), a(s))$  are complex analytic in $s\in \C$.      
Then, as follows from Lemma 2.9 of \cite{PP}, for each $k$,  the product 
$Q(z(s),a(s)) ( \sum_{j=1}^N  Q_{k,j} (z(s), a(s)) (b_j(t,s)-a_j(s))\in \C\{{t,s}\} $.  

Now we consider again $s\in \R$. 
  It follows from the above that  the  denominator  of the fraction in  \eqref{ourpsi3} evaluated on $z(s), a(s)$ is 
real analytic in (one variable) $s\in \R$ and not identically equal to zero.  
The numerator of this fraction, evaluated on $z(s), b(t,s),a(s)$, is analytic in $t\in \K^m$, $s\in \R$, 
$\eta \in \C$.  Hence  $\varphi(t,s)$ is of the form 
$s^{-k}$ times a power series in $(t,s, \eta)$.  Since, moreover,  $\varphi(t,s)$ is bounded  in a neighborhood of the origin, it  has to be analytic.  
 \end{proof}


\section{Construction of the general position deformation}

In this section we produce a stratified submersive family of diffeomorphisms as in section \ref{submersive}. The associated stratification is Zariski equisingular, and the family is constructed inductively by Whitney interpolation.

Consider a family of homogeneous polynomials  with coefficients in $\K =\C$ or $\R$,
 \begin{align}\label{polynomials2}
F_{i} (x_1, \ldots, x_i  )= x_i^{d_i}+ \sum_{j=1}^{d_i} A_{i,j} (x_1, \ldots, x_{i-1}) x_i^{d_i-j}, \quad i=1, \ldots,n, 
\end{align}
satisfying 
\begin{enumerate} 
\item 
 $A_{i,d_i} \equiv 0$ for all $i$ such that $d_i>0$, so that $x_i\equiv 0$ is a root of $F_i$,
\item
for every $i$, the discriminant of $F_{i,red}$ divides   $F_{i-1}$.   
\end{enumerate}
It may happen that $d_i=0$.  Then $F_i\equiv 1$ and we set by convention $F_j\equiv 1$ for $j<i$.   
We call a family $\mathcal F = \{F_i\}$ satisfying the above properties  a \emph{system of 
polynomials}.  This is a special case of the system of pseudopolynomials considered in \S 5 of 
\cite{PP}.  

The \emph{canonical stratification} $\mathcal S$ associated to the system of polynomials \eqref{polynomials2} is defined inductively. For $i=1, \ldots, n$, let us denote by $\pi_{i,i-1}:\K^i\to \K^{i-1}$ the standard projection.  Then the canonical stratification $\mathcal S_i$ of $\K^i$ is obtained from the canonical stratification 
$\mathcal S_{i-1}$ of $ \K^{i-1}$ by lifting its strata to the zero set of $F_i$ and adding, for each stratum $S\in \mathcal S_{i-1}$,  the connected components of $\pi^{-1} (S) \setminus F_i ^{-1} (0)$.   It is defined in   \S 5 of \cite{PP} by means of 
the filtration by the union of strata as follows.  For each $i$ we define a filtration
\begin{equation}\label{filtration}
\K^i = X^i_i\supset X^i_{i-1}\supset\cdots\supset X^i_0,
\end{equation}
where
\begin{enumerate} 
\item 
 $X^1_0= V(F_1)$ (which may be empty),
\item
$X^i_j= (\pi^{-1}_{i,i-1}(X^{i-1}_j)\cap V(F_i))\cup \pi^{-1}_{i,i-1}(X^{i-1}_{j-1})$ for $1\leq j< i$.
\end{enumerate}
Let $\mathcal S_i$ be the stratification associated to the filtration \eqref{filtration}.  Then we set $\mathcal S = \mathcal S_n$.

Following \S 3 of \cite{PP}, we define  a non-trivial deformation of the trivial family 
$F_i(t,x) =F_i(x)$, $t\in \K^n$,  
\begin{align}\label{definitionPhi}
\Ph (t, x) = (t, \Ps (t, x))= (t, \Ps_1(t_1, x_1), \Ps_{2} (t_1, t_2, x_1, x_{2}), \ldots , \Ps_{n} (t, x) ), 
\end{align}
by induction on $n$,
using the interpolation function \eqref{psireal}. Suppose that $\Ps_1,\dots,\Ps_{n-1}$ have been defined. To simplify notation we write $(x_1,\dots,x_n) = (x',x_n)$ and $(t_1\dots,t_n)=(t',t_n)$. From the definition of the family \eqref{polynomials2} and induction, it follows that there is an open neighborhood $U$ of $0$ in $\K^n$ such that the complex roots of $F_n$,
\begin{equation*}
a_1(\Ph_{n-1}(t',x')),\dots, a_{d_n}(\Ph_{n-1}(t',x')),
\end{equation*}
can be chosen $\K$-analytic in $t'$ for all $t\in U$. Moreover, $a_i(0,x') = a_j(0,x')$ if and only if $a_i(\Ph_{n-1}(t',x'))= a_j(\Ph_{n-1}(t',x'))$ for $t\in U$. Let 
\begin{equation*}
a(\Ph_{n-1}(t',x')) = (a_1(\Ph_{n-1}(t',x')),\dots, a_{d_n}(\Ph_{n-1}(t',x')))
\end{equation*}
 be the vector of such roots, and set
\begin{equation}\label{definitionPsi}
\Ps_{n} (t,x) : =  \ps ( x_n, a(0,x'), a(\Ph_{n-1}(t' ,x')), t_n).
\end{equation}

\begin{prop}\label{deformation}
If $U\subset \K^n$ is a sufficiently small neighborhood of the origin, then  
$$\Ph: U \times \K ^n \to U \times \K ^n$$
 is a homeomorphism with the following properties:
\begin{enumerate}
\item
$\Ph$ is semialgebraic, with $\Ph(0,x)=(0,x)$ for all $x\in \K ^n$, and it is an arc-wise analytic trivialization of the projection $U\times \K^n\to U$.
\item
$\Ps$ is $\K^*$-equivariant:
$$
\Ps (t, \lambda x) =\lambda \Ps (t, x) 
$$
for all $x\in \K^n$ and $\lambda\in \K$.  
\item
$\Ph$ preserves the canonical stratification $\mathcal S$ associated to the system of polynomials $\mathcal F$. 
\item
For each stratum $S\in \mathcal S$ the restriction $\Ph : U\times S \to U\times S$ is a real analytic diffeomorphism.  
\item
For each stratum $S\in\mathcal S$ and every $x\in S$ the orbit map $\Ps ^x: U  \to S$ is a $\K$-analytic submersion.  
\end{enumerate}
\end{prop}

\begin{proof}
$\Ph$ is semialgebraic and satisfies $\Ph(0,x)=(0,x)$ for all $x\in \K ^n$ by construction. It is an arc-wise analytic trivialization by Proposition \ref{psiproperties2}.  
The proof follows precisely the arguments of \cite{PP} (p.\ 277, proof of (Z4)). 

Since $F_n$ (and $F_i$ for all $i$) is homogeneous, its roots $a (t,x')$ satisfy 
$a (t,\lambda x')= \lambda a(t,x')$.  Therefore (2) follows from \eqref{lambda}.

Condition (3) follows from the construction (see \cite{PP} \S 5.1, p.\ 282), and (4) follows from (4) of Proposition 
 \ref{psiproperties2}. 

To show (5) we use the inductive constructions of the canonical stratification and the 
deformation $\Ps$.  Let 
 \begin{align}
 \Ps'  (t' , x' )= (\Ps_1(t_1, x_1), \Ps_{2} (t_1, t_2,,x_1, x_{2}), \ldots , \Ps_{n-1} (t', x') ).  
\end{align}
 Let $\mathcal S'$ denote  the canonical stratification of $\K^{n-1}$
 associated to the system of polynomials $\{F_i\}_{i<n} $.  
 By the inductive assumption, there is a neighborhood $U' $ of the origin in $\K^{n-1} $ 
 such that  for every stratum $S' \in \mathcal S'$, 
$\Ps '$ induces a real analytic  map $ \Ps' : U'  \times S' \to S'$ for which  the 
associated  orbit map ${\Ps'} ^{x'}: U'  \to S'$ is submersive for every $x' \in S'$.  
Let $X=F_n^{-1}(0) \subset \K^n$.   
 By construction, the strata of $\mathcal S$ are of two types.  If $S\in \mathcal S$ is of the first type, 
 then $S\subset X$ and  there is $S' \in \mathcal S'$ such that the projection $\pi :\K^n\to \K^{n-1}$ induces  a finite analytic covering $S\to S'$ whose sections are given by the roots of $F_n$.  Since 
 $ \ps ( a_i(0,x'), a(0,x'), a(\Ph_{n-1}(t' ,x')), t_n) =  a_i (\Ph_{n-1}(t' ,x'))$ independently of $t_n$, we see 
 that $\Ps (t,x)$, $x\in S$, does not depend on  $t_n$, and  $ \Ps : U  \times S \to S$ is the lift of 
 $ \Ps' : U'  \times S' \to S'$.  That is, the following diagram commutes:  
  \[
  \xymatrix{
     U'\times S \ar^{\Ps}[rr] \ar_{id\times \pi }[d] && S \ar^{\pi} [d] \\
    U'\times S'  \ar^{\Ps' }[rr]  && S' 
  } 
\]  
Hence $\pi\circ \Ps^x = {\Ps'} ^{\pi(x)}$  is submersive.  

If $S\in \mathcal S$ is of the second type, 
 then there is $S' \in \mathcal S'$ such that $S$ is a connected component of 
 $ \pi^{-1}(S')\cap (\K^n\setminus X)$; that is, $S$ is  open in  $S' \times \K$.   The jacobian matrix 
 $\frac {\partial \Ps } {\partial t}$ is triangular with non-zero terms on the diagonal.  
 Indeed, the latter claim follows from (3) of Proposition \ref{psiproperties} for 
$\frac {\partial \Ps_n } {\partial t_n}$, and by inductive assumption for the other terms.  
\end{proof}

\begin{cor}\label{affine}
The map $\Ps: U \times \K ^n \to  \K ^n$ is a  stratified submersive family of diffeomorphisms.  
\end{cor} 

\subsection{Proof of Theorem \ref{projective}}
Let $g_{ij}\in \K[x_1, \dots ,x_{n+1}]$ be homogeneous polynomials generating the homogeneous ideals defining $V_i\subset \proj^n$.  Let $F_{n+1} (x_1, \dots ,x_{n+1})$ be the product of all $g_{ij}$.  We complete $F_{n+1}$ to a system of polynomials and apply the above construction of the deformation $\Ps $.  Thanks to property (2) of Proposition \ref{deformation}, $\Ps$ can be projectivised to a map  $U \times \proj ^{n}  \to \proj ^{n} $.
By Proposition 1.9 of \cite{PP},  $\Ps$ preserves the zero sets of each $g_{ij}$   and hence 
each $V_i$.  
 This shows Theorem \ref{projective}.
\qed


\end{document}